\documentclass[12pt,twoside]{amsart}
\usepackage{amsmath, amsthm, amscd, amsfonts, amssymb, graphicx}
\newtheorem{thm}{Theorem}[section]
\newtheorem{cor}[thm]{Corollary}
\newtheorem{lem}[thm]{Lemma}

\newtheorem{rem}[thm]{\bf{Remark}}

\numberwithin{equation}{section}

\begin{document}
\begin{center}
{\bf{Some properties of Generalized Fibonacci difference bounded and $p$-absolutely convergent sequences}}
\\Anupam Das  and Bipan Hazarika$^\ast$


Department of Mathematics, Rajiv Gandhi University, Rono Hills, Doimukh-791 112, Arunachal Pradesh, India\\
Email: anupam.das@rgu.ac.in; bh\_rgu@yahoo.co.in 


\end{center}

\title{}
\author{}
\thanks{{03-01-2016 \\
$^\ast$The corresponding author}}





\begin{abstract} 
The main objective of this paper is to introduced a new sequence space $l_{p}(\hat{F}(r,s)),$ $ 1\leq p \leq \infty$ by using the band matrix $\hat{F}(r,s).$  
 We also establish a few inclusion relations concerning this space and determine its $\alpha-,\beta-,\gamma-$duals. We also characterize some matrix classes on the space $l_{p}(\hat{F}(r,s))$ and examine some geometric properties of this space.
 
\vskip 0.5cm


\vskip 0.5cm

\textbf{AMS Subject classification no:} 11B39; 46A45; 46B45; 46B20. 
\end{abstract}

\maketitle
\pagestyle{myheadings}
\markboth{\rightline {\scriptsize Das and Hazarika}}
         {\leftline{\scriptsize Some properties of  Generalized Fibonacci...}}

\maketitle

\section{ Introduction}
\normalfont

Let $\omega$ be the space of all real-valued sequences. Any vector subspace of $\omega$ is called a $\mathit{sequence\ space}.$ By $l_{\infty},c ,c_{0}$ and $l_{p} \ (1\leq p <  \infty),$ we denote the sets of all bounded, convergent, null sequences and $p-$absolutely convergent series, respectively. Also we use the convensions that $e=(1,1,...)$ and $e^{(n)}$ is the sequence whose only non-zero term is 1 in the $nth$ place for each $n\in\mathbb{N}, $ where $\mathbb{N}=\left\lbrace  0,1,2,...\right\rbrace .$

Let $X$and $Y$ be two sequence spaces and $A=(a_{nk})$ be an infinite matrix of real numbers $a_{nk},$ where $n,k\in\mathbb{N}.$ We write $A=(a_{nk})$ instead of $A=(a_{nk})_{n,k=0}^{\infty}.$ Then we say that $A$ defines a matrix mapping from $X$ into $Y$ and we denote it by writing $A:X\rightarrow Y$ if for every sequence $x=(x_{k})_{k=0}^{\infty}\in X,$ the sequence $Ax=\left\lbrace A_{n}(x) \right\rbrace_{n=0}^{\infty} ,$ the $A$-transform of $x,$ is in $Y,$ where
\begin{equation}
 A_{n}(x)=\sum\limits_{k=0}^{\infty} a_{nk}x_{k}\ \left( n\in \mathbb{N}\right).
 \label{11}
\end{equation}

For simplicity in notation, here and in what follows, the summation without limits runs from $0$ to $\infty.$ Also if $x\in\omega,$ then we write $x=\left(x_{k} \right)_{k=0}^{\infty}.$\par
By $(X,Y),$ we denote the class of all matrices $A$ such that $A:X\rightarrow Y.$ Thus $A\in (X,Y)$ iff the series on the right-hand side of (\ref{11}) converges for each $n\in\mathbb{N} $ and every $x\in X$ and we have $Ax\in Y$ for all  $x\in X.$ 
\par The approach constructing a new sequence space by means of matrix domain has recently employed by several authors.\par
The matrix domain $X_{A}$ of an infinite matrix $A$ in a sequence space $X$ is defined by
\[X_{A}=\left\lbrace x=(x_{k})\in \omega: Ax \in X\right\rbrace.
\]
\par Let $\Delta$ denote the matrix $\Delta=(\Delta_{nk})$ defined by
$$
\Delta_{nk}
= \left\{
        \begin{array}{ll}
           (-1)^{n-k} , 
           & n-1\leq k \leq n \\
            0 ,
           & 0 \leq k < n-1\quad or\quad k>n
        \end{array}
    \right.
$$
The concept of matrix domain we refer to \cite{A,Altay,Aydin,Kiri,Mishra,M.Mursaleen,M,Mursaleen,M1,Michael,Wilansky}.\\
Define the sequence $\left\lbrace f_{n}\right\rbrace_{n=0}^{\infty}$
of Fibonacci numbers given by the linear recurrence relations $ f_{0}=f_{1}=1$ and $ f_{n}=f_{n-1}+f_{n-2}, n\geq 2.$\\
Fibonacci numbers have many interesting properties and applications. For example, the ratio sequences of Fibonacci numbers converges to the golden ratio which is important in sciences and arts. Also some basic properties of Fibonacci numbers are given as follows:
\[\lim_{n\rightarrow \infty}\frac{f_{n+1}}{f_{n}}= \frac{1+\sqrt{5}}{2}=\alpha\quad
(golden\ ratio),
\]
\[\sum_{k=0}^{n}f_{k}=f_{n+2}-1 \quad (n\in \mathbb{N}),
\]
\[\sum_{k}\frac{1}{f_{k}} \mbox{~converges,}
\]
\[f_{n-1}f_{n+1}-f_{n}^{2}=(-1)^{n+1} \quad (n \geq 1)  (Cassini formula)
\]
Substituting for $ f_{n+1} $ in Cassini's formula yields $ f_{n-1}^{2}+f_{n}f_{n-1}-f_{n}^{2}=(-1)^{n+1}.$ For the properties of Fibonnaci numbers and matrix domain related to Fibonnaci numbers we refer to \cite{Kara,T.K}.
\par A sequence space $X$ is called a $FK-$space if it is complete linear metric space with continuous coordinates $ p_{n}:X\rightarrow \mathbb{R} (n\in  \mathbb{N} ),$  where $ \mathbb{R}$ denotes the real field and $ p_{n}(x)=x_{n} $  for all $ x=(x_{k}) \in X $ and every $ n \in \mathbb{N} .$ A $BK-$ space  is a normed $FK-$ space, that is a $BK-$space  is a Banach space with continuous coordinates. The space $ l_{p}(1\leq p < \infty)$ is a BK-space with the norm
 \[ \parallel x \parallel_{p} =\left(\sum\limits_{k=0}^{\infty}\mid x_{k} \mid ^{p}\right)^{1/p}\]
and $ c_{0},c $ and $ l_{\infty} $ are BK-spaces with the norm
 \[\parallel x \parallel_{\infty}
=\sup_{k}| x_{k}|.\]
 The sequence space $ \lambda $ is said to be solid if and only if 
\[\tilde{\lambda}=\left\lbrace \left( u_{k} \right) \in \omega : \exists\left( x_{k} \right) \in \lambda  \mbox{~such~that~} \mid u_{k}\mid \leq \mid x_{k} \mid  \forall k \in \mathbb{N} \right\rbrace \subset \lambda.\]
 A sequence $(b_{n})$ in a normed space $X$ is called a $ Schauder \ basis $ for $X$ if every $x\in  X,$ there is a unique sequence $(\alpha_{n})$
of scalars such that $x=\sum_{n}\alpha_{n}b_{n}, $ i.e., $\lim\limits_{m \rightarrow \infty}\parallel x-\sum\limits_{n=0}^{m}\alpha_{n}b_{n}\parallel =0 .$
\par The  $\alpha-,\beta-,\gamma-$duals of the sequence space $X$ are respectively defined by \\
$X^{\alpha}=\left\lbrace a=(a_{k})\in \omega : ax=(a_{k}x_{k})\in l_{1}\:\forall\: x=(x_{k})\in X  \right\rbrace, $\\
$X^{\beta}=\left\lbrace a=(a_{k})\in \omega : ax=(a_{k}x_{k})\in cs \:\forall\: x=(x_{k})\in X  \right\rbrace ,$\\
and \\
$X^{\gamma}=\left\lbrace a=(a_{k})\in \omega : ax=(a_{k}x_{k})\in bs \:\forall\: x=(x_{k})\in X  \right\rbrace ,$ \\
where $cs$ and $bs$ are the sequence spaces of all convergent and bounded series, respectively (see for instance \cite{A,M,PK}).\\ 
 Now let $A=(a_{nk})$ be an infinite matrix and consider  the following conditions:
\begin{equation}
\sup_{n} \sum\limits_{k}\left| a_{nk}\right|^{q}  < \infty, q=\frac{p}{p-1}
\label{17}
\end{equation} 
\begin{equation}
\lim_{n}a_{nk}\: \mbox{~exists~}\forall \:k
\label{18}
\end{equation}
\begin{equation}
\sup_{K \in \mathcal{F}} \sum\limits_{k}\left| \sum\limits_{n \in K} a_{nk}\right|^{q}  < \infty, q=\frac{p}{p-1}
\label{19}
\end{equation}
\begin{equation}
\lim\limits_{n}\sum\limits_{k}\left| a_{nk}\right|=
\sum\limits_{k}\left|\lim\limits_{n} a_{nk}\right|
\label{20}
\end{equation}
Now we may give the following lemma due to Stieglitz and Tietz \cite{Michael} on the characterization of the matrix transformations between some sequence spaces.
\begin{lem}\label{lem11}
The following statements hold:
\begin{enumerate}
\item[(a)] $A =(a_{nk}) \in (l_{p},c)$ iff (\ref{17}),(\ref{18}) holds, $1< p < \infty .$
\item[(b)] $A =(a_{nk}) \in (l_{p},l_{1})$ iff (\ref{19}) holds, $1 < p < \infty.$
\item[(c)] $A =(a_{nk}) \in (l_{\infty},c)$ iff (\ref{18}),(\ref{20}) holds.
\item[(d)] $A =(a_{nk}) \in (l_{p},l_{\infty})$ iff
(\ref{17}) holds, $1 < p < \infty.$
\end{enumerate}
\end{lem} 
\section{Fibonacci difference sequence space $l_{p}(\hat{F}(r,s))$ }

In this section, we have used the Fibonacci band matrix  $ \hat{F}(r,s)=\left( f_{nk}(r,s) \right) $ and introduce the sequence space $l_{p}(\hat{F}(r,s)).$ Also we present some inclusion theorems and construct the Schauder basis of the space $l_{p}(\hat{F}(r,s)).$\par
Let $f_{n}$ be the $nth$ Fibonacci number for every $n \in \mathbb{N}.$ Then we define the infinite matrix  $\hat{F}(r,s)=\left( f_{nk}(r,s) \right) $ by
\begin{equation}
f_{nk}(r,s)
= \left\{
        \begin{array}{ll}
           s\frac{f_{n+1}}{f_{n}} , 
           & k=n-1 \\
            r\frac{f_{n}}{f_{n+1}} ,
           & k=n \\
           0,
           & 0\leq k < n-1\quad or\quad k>n
        \end{array}
    \right.
\end{equation}
where $n,k \in \mathbb{N}$ and $r,s \in \mathbb{R}- \left\lbrace 0\right\rbrace.$

\par Define the sequence $y=(y_{n}),$ which will be frequently used, by the $\hat{F}(r,s)$ - transform of a sequence $x=(x_{n}),$ i.e., $y_{n}=\hat{F}(r,s)_{n}(x),$ where
\begin{equation}
y_{n}
= \left\{
        \begin{array}{ll}
           r\frac{f_{0}}{f_{1}}x_{0}=rx_{0} , 
           &  n=0 \\
            r\dfrac{_{f_{n}}}{f_{n+1}}x_{n}+s\dfrac{_{f_{n+1}}}{f_{n}}x_{n-1} ,
           &  n \geq 1
        \end{array}
    \right.
  \label{22}  
\end{equation}
where $ n \in \mathbb{N} $ .
\par Moreover it is obvious that $\hat{F}(r,s)$ is a triangle. Thus it has a unique inverse $\hat{F}(r,s)^{-1}=(\hat{f}_{nk}(r,s)^{-1})$ and it is given by
\begin{equation}
\hat{f}_{nk}(r,s)^{-1}
= \left\{
        \begin{array}{ll}
          \frac{1}{r}\left( -\frac{s}{r}\right)^{n-k}  \frac{f_{n+1}^{2}}{f_{k}f_{k+1}} , 
           &  0 \leq k \leq n \\
           0 ,
           &  k>n 
        \end{array}
    \right.
   \label{23} 
\end{equation}
for all $ n,k \in \mathbb{N}.$ There we have by (\ref{23}) that 
\begin{equation}
x_{k}=\sum\limits_{j=0}^{k}   \frac{1}{r}\left( -\frac{s}{r}\right)^{k-j}  \frac{f_{k+1}^{2}}{f_{j}f_{j+1}}y_{j}; (k \in \mathbb{N}).
\label{24}
\end{equation}
\par Now we introduce new Fibonacci sequence spaces as follows 
\[l_{p}(\hat{F}(r,s))=\left\lbrace x=(x_{n})\in\omega:\sum\limits_{n=1}^{\infty}\left| r\dfrac{_{f_{n}}}{f_{n+1}}x_{n}+s\dfrac{_{f_{n+1}}}{f_{n}}x_{n-1}\right|^{p}<\infty \right\rbrace, 1\leq p< \infty \] and
\[l_{\infty}(\hat{F}(r,s))=\left\lbrace x=(x_{n})\in\omega:\sup_{n}\left| r\dfrac{_{f_{n}}}{f_{n+1}}x_{n}+s\dfrac{_{f_{n+1}}}{f_{n}}x_{n-1}\right|<\infty \right\rbrace .\]
\par The sequence spaces $l_{p}(\hat{F}(r,s))$ and $l_{\infty}(\hat{F}(r,s))$ may be redefined as 
\begin{equation}
l_{p}(\hat{F}(r,s))=\left( l_{p}\right)_{\hat{F}(r,s)} ,
l_{\infty}(\hat{F}(r,s))=\left( l_{\infty}\right)_{\hat{F}(r,s)}.
\label{25}
\end{equation}
\par In this section, we give some results related to the space  $l_{p}(\hat{F}(r,s)), 1\leq p \leq \infty.$
\begin{thm}
Let $1 \leq p< \infty.$ Then $l_{p}(\hat{F}(r,s))$ is a BK-space with norm 
  \[ \parallel x\parallel_{l_{p}(\hat{F}(r,s))}= \left( \sum\limits_{k}\left| \hat{F}(r,s)_{k}(x)\right|^{p} \right)^{1/p}  \]
  and $l_{\infty}(\hat{F}(r,s))$ is a BK-space with norm 
  \[ \parallel x\parallel_{l_{\infty}(\hat{F}(r,s))}= \sup_{k}| \hat{F}(r,s)_{k}(x)|.\]
\end{thm} 

\begin{proof}
Since (\ref{25}) holds, $l_{p}$ and $l_{\infty}$ are BK-spaces with respect to their natural norm and the matrix $\hat{F}(r,s)$ is triangular matrix. By Theorem 4.3.3 of Wilansky \cite{Wilansky} gives the fact that the spaces 
$l_{p}(\hat{F}(r,s)),1 \leq p< \infty $ and   $l_{\infty}(\hat{F}(r,s))$ 
are BK space with the given norms. 
\end{proof}    
\begin{rem}
The spaces $l_{p}(\hat{F}(r,s))$ for $ 1\leq p < \infty$ and $l_{\infty}(\hat{F}(r,s))$ are non-absolute type because $ \parallel x \parallel_{l_{p}(\hat{F}(r,s))} \neq \parallel \mid x \mid \parallel_{l_{p}(\hat{F})(r,s)} $ and
 $ \parallel x \parallel_{l_{\infty}(\hat{F}(r,s))} \neq \parallel \mid x \mid \parallel_{l_{\infty}(\hat{F})(r,s)},$ where $ \mid x \mid =(\mid x_{k} \mid).$ 
\end{rem} 
\begin{thm}\label{thm23}
The sequence spaces  $l_{p}(\hat{F}(r,s)), 1 \leq p < \infty$ and $l_{\infty}(\hat{F}(r,s))$ of non-absolute type are linearly isomorphic to the spaes $l_{p}$ and $l_{\infty},$ respectively, i.e. $l_{p}(\hat{F}(r,s)) \cong l_{p}$ and $l_{\infty}(\hat{F}(r,s)) \cong l_{\infty}.$
\end{thm}
\begin{proof}
To prove this, we have to show that there exists a linear bijective mapping between  $l_{p}(\hat{F}(r,s))$ and $l_{p}$ for $1 \leq p \leq \infty.$\\
 Let us consider a mapping $T$ defined from  $l_{p}(\hat{F}(r,s))$ to $l_{p}$ by $Tx=\hat{F}(r,s)(x)=y \in l_{p}$ for every $x \in l_{p}(\hat{F}(r,s)),$ where $x=(x_{k})$ and $y=(y_{k}).$ 
\par It is obvious that $T$ is linear. Further, it is trivial that $x=0$ whenever $Tx=0.$ Hence $T$ is injective.
\par Let $y=(y_{k}) \in l_{p},$ $ 1 \leq p \leq \infty$ and define the sequence $x=(x_{k})$ by 
\[x_{k}=\sum\limits_{j=0}^{k}  \frac{1}{r}\left(- \frac{s}{r}\right)^{k-j} \frac{f_{k+1}^{2}}{f_{j}f_{j+1}}y_{j} \mbox{~for ~all~} k \in \mathbb{N}.\]
Then, in the cases $1 \leq p < \infty$ and $p=\infty$ we get\\
$\parallel  x \parallel_{l_{p}(\hat{F}(r,s))}
 = \left( \sum\limits_{k}\left| r\frac{f_{k}}{f_{k+1}}x_{k}+s\frac{f_{k+1}}{f_{k}}x_{k-1}\right|^{p}\right) ^{1/p}\\
=\left( \sum\limits_{k}\left| r\frac{f_{k}}{f_{k+1}}\sum\limits_{j=1}^{k}\frac{1}{r}\left(-\frac{s}{r} \right)^{k-j}\frac{f_{k+1}^{2}}{f_{j}f_{j+1}}y_{j} +s\frac{f_{k+1}}{f_{k}}\sum\limits_{j=1}^{k-1}\frac{1}{r}\left(-\frac{s}{r} \right)^{k-j-1}\frac{f_{k+1}^{2}}{f_{j}f_{j+1}}y_{j}\right|^{p}\right) ^{1/p}
\\=\left(\sum\limits_{k}\mid y_{k}\mid ^{p}\right)^{1/p} 
\\=\parallel  y \parallel_{l_{p}}
< \infty.$\\
Similarly we can show that
$\parallel  x \parallel_{l_{\infty}(\hat{F}(r,s))}=\parallel y \parallel_{\infty}.$\\
Thus we have $x \in l_{p}(\hat{F}(r,s))$ for $ 1 \leq p \leq \infty.$ Hence $T$ is surjective and norm preserving. Consequently $T$ is a linear bijection which proves that the spaces $l_{p}(\hat{F})(r,s)$ and $l_{p}$  are linearly isomorphic for $ 1 \leq p \leq \infty.$
\end{proof}
\begin{thm}
 $l_{p} \subset l_{p}(\hat{F}(r,s))$ holds for $ 1\leq p \leq \infty$ and for finite $r,s$ such that $\left| -\frac{s}{r} \right| \geq 1,$
  $\mid r \mid \leq 1$ and $\mid s \mid \leq 1/2.$
\end{thm} 
\begin{proof}
Let $ x=(x_{k}) \in l_{p} $ and $1\leq p \leq \infty .$ 
Since the inequalities  $ \frac{f_{k}}{f_{k+1}} \leq 1 $ and $ \frac{f_{k+1}}{f_{k}} \leq 2 $ for every $ k \in \mathbb{N} $ therefore we have \\
$ \sum\limits_{k}\mid \hat{F}(r,s)_{k}(x) \mid^{p} \\
=\sum\limits_{k}\left| r\frac{f_{k}}{f_{k+1}}x_{k}+s\frac{f_{k+1}}{f_{k}}x_{k-1}\right|^{p}
\\
\leq \mid r\mid^{p}\sum\limits_{k}\mid x_{k}\mid ^{p}+ \mid 2s\mid^{p}\sum\limits_{k}\mid x_{k-1}\mid ^{p} $\\
and
\\
$\sup_{k \in \mathbb{N}}\mid \hat{F}(r,s)_{k}(x)\mid
\leq (\mid r \mid + \mid 2s\mid)\sup_{k \in \mathbb{N}}\mid x_{k}\mid$\\
which together gives\\
$\parallel x\parallel_{l_{p}(\hat{F}(r,s))}\leq (\mid r \mid + \mid 2s\mid)\parallel x\parallel_{l_{p}}$ for $1\leq p \leq \infty ,$ where $r,s$ are finite.\\
Therefore $ \parallel x \parallel_{l_{p}(\hat{F})(r,s)} < \infty,$ since
$  x \in l_{p}. $ \\
Hence $ l_{p} \subseteq l_{p}(\hat{F}(r,s)). $ 
Further since $x=(x_{k})=\left(\frac{1}{r}\left( -\frac{s}{r}\right)^{k}  f_{k+1}^{2}\right) $ is in $l_{p}(\hat{F}(r,s))-l_{p}$ for 
$\left| -\frac{s}{r} \right| \geq 1.$
Therefore $l_{p} \subset l_{p}(\hat{F}(r,s))$ for $1 \leq p \leq \infty.$
\end{proof}
\begin{thm}
For $1 \leq p < q,$
 $ l_{p}(\hat{F}(r,s)) \subset l_{q}(\hat{F}(r,s))$ holds.
\end{thm}
\begin{proof}
Let $1 \leq p < q$ and $x \in  l_{p}(\hat{F}(r,s)).$ Then we obtain from Theorem \ref{thm23} that $y \in l_{p},$ where  $y=\hat{F}(r,s)(x).$ We have $l_{p} \subset l_{q}$ which gives 
$y \in l_{q}.$ This means that $x \in  l_{q}(\hat{F}(r,s)).$ Hence we have 
$ l_{p}(\hat{F}(r,s)) \subset l_{q}(\hat{F}(r,s)).$
\end{proof}
\begin{thm}
If $\left| -\frac{s}{r} \right| \geq 1$ then the space $l_{\infty}$ does not include the space $ l_{p}(\hat{F}(r,s)).$
\end{thm}
\begin{proof}
Let $\left| -\frac{s}{r} \right| \geq 1$ and $x=(x_{k})= \left(\frac{1}{r}\left( -\frac{s}{r}\right)^{k}  f_{k+1}^{2} \right). $ We know that $ f_{k+1}^{2} \rightarrow \infty$ as $k \rightarrow \infty$ and  $\hat{F}(r,s)(x)=(1,0,0,0,...).$ Therefore the sequence lies in
$ l_{p}(\hat{F}(r,s))$ but not in $l_{\infty}.$ This completes the proof. 
\end{proof}
\begin{thm}
If $\left| -\frac{s}{r} \right| \geq 1$ then the space $bv_{p}$ does not include the space $ l_{p}(\hat{F}(r,s)).$
\end{thm}
\begin{proof}
Let $\left| -\frac{s}{r} \right| \geq 1$ and $x=(x_{k})= \left(\frac{1}{r}\left( -\frac{s}{r}\right)^{k}  f_{k+1}^{2} \right) .$ We know that $ f_{k+1}^{2} \rightarrow \infty$ as $k \rightarrow \infty$ and  $\hat{F}(r,s)(x)=(1,0,0,0,...)$ and
$\Delta x=(\Delta x_{k})=
\left( -\frac{1}{r}\left(  -\frac{s}{r}\right)^{k-1}\left( 
\frac{s}{r}f_{k+1}^{2}+f_{k}^{2} \right)   \right). $ Clearly for
$\left| -\frac{s}{r} \right| \geq 1,$ $\Delta x \notin l_{p}.$
Therefore the sequence lies in
$ l_{p}(\hat{F}(r,s))$ but not in $bv_{p}.$ This completes the proof. 
\end{proof}
\begin{lem} \cite{A}
Let $\lambda$ be a BK-space including the space $\phi.$ Then
$\lambda$ is solid if and only if $l_{\infty}\lambda \subset \lambda.$
\end{lem}  
\begin{thm}
 The space $l_{p}(\hat{F}(r,s)),1 \leq p \leq \infty $ is solid.
 \begin{proof}
 Let the sequences $x=(x_{k}) \in l_{p}(\hat{F}(r,s))$ and $y=(y_{k}) \in l_{\infty}.$ We have the following conditions
 \[ \sum_{k}\left|r\frac{f_{k}}{f_{k+1}}x_{k}+s\frac{f_{k+1}}{f_{k}}x_{k-1}\right|^{p} < \infty  \mbox{~for~} 1 \leq p < \infty\] and 
 $\sup_{k}\mid y_{k}\mid < \infty.$\\
  So there exists a non-negative real number $M$ such that 
 $\sup_{k}\mid y_{k}\mid =M.$
 \par 
 Consider any element $yx=(y_{k}x_{k})$ $\in l_{\infty}l_{p}(\hat{F}(r,s)).$\\ Now
 \[\sum_{k}\left|r\frac{f_{k}}{f_{k+1}}y_{k}x_{k}+s\frac{f_{k+1}}{f_{k}}y_{k-1}x_{k-1}\right|^{p}
 \leq M^{p} \sum_{k}\left| r\frac{f_{k}}{f_{k+1}}x_{k}+s\frac{f_{k+1}}{f_{k}}x_{k-1}\right|^{p} < \infty.\]
 Therefore $yx \in l_{p}(\hat{F}(r,s)),$ which implies $l_{\infty}l_{p}(\hat{F}(r,s)) \subset l_{p}(\hat{F}(r,s))$ for $1 \leq p < \infty .$ Similarly, we can show that 
 $l_{\infty}l_{\infty}(\hat{F}(r,s)) \subset l_{\infty}(\hat{F}(r,s)),$ which completes the proof.
 \end{proof}
\end{thm} 

\par Now we give a sequence of points of the space $l_{p}(\hat{F}(r,s))$ which will form the basis for the space
$l_{p}(\hat{F}(r,s))$ for $1 \leq p < \infty.$
\begin{thm}
Let $1 \leq p < \infty $ and define the sequence $ c^{(n)} \in l_{p}(\hat{F}(r,s)) $ for every fixed $ n \in \mathbb{N}$ by
\begin{equation}
(c^{(n)})_{k}
= \left\{
        \begin{array}{ll}
           0 , 
           &  0 \leq k \leq n-1\\
            \frac{1}{r}.\left( -\frac{s}{r}\right)^{k-n}. \frac{f_{k+1}^{2}}{f_{n}f_{n+1}} ,
           &  k \geq n
        \end{array}
    \right.
  \label{26}  
\end{equation}
 where $ n \in \mathbb{N}.$ Then the sequence $ (c^{(n)})_{n=0}^{\infty}$ is a basis for the space $ l_{p}(\hat{F}(r,s)),$ and every $  x \in  l_{p}(\hat{F}(r,s))$ has a unique representation of the form 
\begin{equation}
 x= \sum\limits_{n} \hat{F}(r,s)_{n}(x)c^{(n)}.
 \label{27}
 \end{equation}
\end{thm}
\begin{proof}
Let $ 1 \leq p < \infty .$ It is obvious by that $ \hat{F}(r,s)(c^{(n)})=e^{(n)} \in l_{p}\;(k \in \mathbb{N})$ and hence $ c^{(n)} \in l_{p}(\hat{F}(r,s)) $ for all $ k \in \mathbb{N}.$
\par Further, let $ x \in l_{p}(\hat{F}(r,s)) .$ For any non-negative integer $ m,$ we put 
$ x^{(m)}= \sum\limits_{n=0}^{m} \hat{F}(r,s)_{n}(x)c^{(n)}.$
\par Then we have that
\[ \hat{F}(r,s)(x^{(m)}) = \sum\limits_{n=0}^{m}\hat{F}(r,s)_{n}(x)\hat{F}(r,s)(c^{(n)})= \sum\limits_{n=0}^{m}\hat{F}(r,s)_{n}(x)e^{(n)} \]
and hence  \[ \hat{F}(r,s)_{k}(x-x^{(m)})
= \left\{
        \begin{array}{ll}
           0 , 
           & 0 \leq k \leq m \\
            \hat{F}(r,s)_{k}(x)  ,
           &  k > m
        \end{array}
    \right. \] where $ k,m \in \mathbb{N}.$\\
 For any given $\epsilon >0,$ there is a non-negative integer $ m_{0}$ such that \[\sum\limits_{n=m_{0}+1}^{\infty}\left| \hat{F}(r,s)_{n}(x)\right|^{p} \leq \left( \frac{\epsilon}{2} \right)^{p}.\]
Therefore we have for every $ m \geq m_{0} $ that
\[ \parallel x-x^{(m)}\parallel_{l_{p}(\hat{F}(r,s))}\]
\[ =  
\left( \sum\limits_{n=m+1}^{\infty}\left|\hat{F}(r,s)_{n}(x)\right|^{p}\right)^{1/p}\]
\[ \leq
\left( \sum\limits_{n=m_{0}+1}^{\infty}\left|\hat{F}(r,s)_{n}(x)\right|^{p}\right)^{1/p} \leq \frac{\epsilon}{2}
 < \epsilon \]
which shows that
 $ \lim\limits_{m \rightarrow \infty}\parallel x-x^{(m)}\parallel_{l_{p}(\hat{F}(r,s))}=0 $ and hence $x$ is represented as in (\ref{27}).
 \par Now we are going to show the uniqueness of the representation (\ref{27}) of $x \in l_{p}(\hat{F}(r,s)).$ Let $x= \sum\limits_{k} \mu_{k}(x)c^{(k)}.$  We have $  \hat{F}(r,s)$ is a linear mapping from $ l_{p}(\hat{F}(r,s))$ to $l_{p}.$ Since any matrix mapping between FK spaces is continuous, so $  \hat{F}(r,s)$ is continuous.\\
 Now 
 \[\hat{F}(r,s)_{n}(x)=\sum\limits_{k}\mu_{k}(x)\hat{F}(r,s)_{n}(c^{(k)})=\mu_{n}(x)\quad (n\in \mathbb{N}).\] 
 Hence the representation (\ref{27}) is unique.
\end{proof}

\section{ The $\alpha$-,$\beta$- and $\gamma$-duals of the space $l_{p}(\hat{F}(r,s))$}
In this section, we determine the $\alpha-,$ $\beta-$ and $\gamma-$duals of the sequence space $l_{p}(\hat{F}(r,s)).$ Since the case $p=1$ can be proved by analogy, we omit the proof of that case and consider only the case $1 <p \leq \infty.$
\begin{thm}
The $\alpha$-dual of the sequence space 
$ l_{p}(\hat{F}(r,s))$ is the set \\
$ d_{1}=\left\lbrace a=(a_{k}) \in \omega :  
\sup_{K \in \mathcal{F}} \sum\limits_{k}\left| \sum\limits_{n \in K} b_{nk}\right|^{q}  < \infty, 
q=\frac{p}{p-1}
 \right\rbrace $
  where $1<p\leq \infty$ and the matrix $B=(b_{nk})$ 
is defined as follows \\
 $$ b_{nk}
= \left\{
        \begin{array}{ll}
          \frac{1}{r}\left( -\frac{s}{r}\right)^{n-k}  \frac{f_{n+1}^{2}}{f_{k}f_{k+1}}a_{n} , 
           & 0 \leq k \leq n \\
            0  ,
           &  k > n
        \end{array}
  \right. $$ for all $n,k \in \mathbb{N}$
  and $a=(a_{n}) \in \omega.$
\end{thm}
\begin{proof}
Let $a=(a_{n}) \in \omega.$ Also for every $ x=(x_{n})\in \omega, $ we put $y=(y_{n})=\hat{F}(r,s)(x).$ Then it follows by (\ref{24}) that $x_{k}=\sum\limits_{j=0}^{k}\frac{1}{r}\left( -\frac{s}{r}\right)^{k-j}\frac{f_{k+1}^{2}}{f_{j}f_{j+1}}y_{j} $ and 
\begin{equation}
B_{n}(y)=\sum\limits_{k=0}^{n}b_{nk}y_{k}=
\sum\limits_{k=0}^{n}\frac{1}{r}\left( -\frac{s}{r}\right)^{n-k}\frac{f_{n+1}^{2}}{f_{k}f_{k+1}}a_{n}y_{k}=a_{n}x_{n} .
\label{31}
\end{equation}
where $ n \in \mathbb{N}.$
\par Thus we observe by (\ref{31}) that $ax=(a_{n}x_{n})\in l_{1}$ whenever $x \in l_{p}(\hat{F}(r,s))$ if and only if $By \in l_{1}$ whenever  $y \in l_{p}.$ Therefore we derive by using the Lemma \ref{lem11} that\\
$ \sup_{K \in \mathcal{F}} \sum\limits_{k}\left| \sum\limits_{n \in K} b_{nk}\right|^{q}<\infty$ which implies that $\left\lbrace l_{p}(\hat{F}(r,s))\right\rbrace^{\alpha}=d_{1}. $
\end{proof}     
\begin{thm}
Define the sets $d_{2},d_{3}$ and $d_{4}$ by \\
$ d_{2}=\left\lbrace a=(a_{k}) \in \omega : \sup_{n} \sum\limits_{k} \mid d_{nk}\mid ^{q} < \infty , q=\frac{p}{p-1}  \right\rbrace, $ \\
$ d_{3}=\left\lbrace a=(a_{k}) \in \omega : \lim_{n}d_{nk}\mbox{~exists~} \forall \: k\right\rbrace, $
\\ and
$ d_{4}=\left\lbrace a=(a_{k}) \in \omega : \lim_{n} \sum\limits_{k=0}^{n} \mid d_{nk} \mid =
\sum\limits_{k} \mid \lim_{n}d_{nk} \mid \right\rbrace. $ \\
 
 Then $ \left\lbrace  l_{p}(\hat{F}(r,s))\right\rbrace ^{\beta}=d_{2}\cap d_{3} $ and 
 $ \left\lbrace  l_{\infty}(\hat{F}(r,s))\right\rbrace ^{\beta}=d_{2}\cap d_{4}$ where $1 < p < \infty$ and $D=(d_{nk})$ is defined by \\
  $$ d_{nk}
 = \left\{
         \begin{array}{ll}
           \sum\limits_{j=k}^{n}\frac{1}{r}\left( -\frac{s}{r}\right)^{j-k}  \frac{f_{j+1}^{2}}{f_{k}f_{k+1}}a_{n} , 
            & 0 \leq k \leq n \\
             0  ,
            &  k > n
         \end{array}
   \right. $$ for all $n,k \in \mathbb{N}.$
\end{thm}     
\begin{proof}
Let $a=(a_{k})\in \omega$ and consider the equality 
\begin{equation}
\sum\limits_{k=0}^{n} a_{k}x_{k}
=\sum\limits_{k=0}^{n} a_{k}\left( \sum\limits_{j=0}^{k}\frac{1}{r}\left( -\frac{s}{r}\right)^{k-j}  \frac{f_{k+1}^{2}}{f_{j}f_{j+1}}y_{j}
 \right) 
=\sum\limits_{k=0}^{n}\left( \sum\limits_{j=k}^{n}\frac{1}{r}\left( -\frac{s}{r}\right)^{j-k} \frac{f_{j+1}^{2}}{f_{k}f_{k+1}}a_{j} \right)y_{k} =D_{n}(y)
\label{32}
\end{equation}
where $D=(d_{nk})$ is defined by
$$ d_{nk}
 = \left\{
         \begin{array}{ll}
           \sum\limits_{j=k}^{n}\frac{1}{r}\left( -\frac{s}{r}\right)^{j-k}  \frac{f_{j+1}^{2}}{f_{k}f_{k+1}}a_{n} , 
            & 0 \leq k \leq n \\
             0  ,
            &  k > n
         \end{array}
   \right. $$ where $ n,k \in \mathbb{N}.$ Then we deduce from Lemma \ref{lem11} that $ ax=(a_{k}x_{k}) \in cs$ whenever $x=(x_{k}) \in l_{p}(\hat{F}(r,s))$ if and only if $Dy \in c $ whenever $y \in l_{p}.$ Thus $a \in\left\lbrace  l_{p}(\hat{F}(r,s))\right\rbrace ^{\beta} $ if and only if $a \in d_{2},$ $a \in d_{3}.$ Hence
    $\left\lbrace l_{p}(\hat{F}(r,s))\right\rbrace ^{\beta}=d_{2}\cap d_{3}.$
   Similarly, we can show that 
    $\left\lbrace l_{\infty}(\hat{F}(r,s))\right\rbrace ^{\beta}=d_{3}\cap d_{4}.$
\end{proof}
\begin{thm}
 $\left\lbrace l_{p}(\hat{F}(r,s))\right\rbrace ^{\gamma}=d_{2},1<p \leq \infty.$
\end{thm}
\begin{proof}
This result can be obtained from Lemma \ref{lem11}.
\end{proof}
\section{ Some matrix transformations related to the sequence space $ l_{p}(\hat{F}(r,s))$ }

In this section, we characterize the classes $ \left( l_{p}(\hat{F}(r,s)),X \right),$ where $1 \leq p \leq \infty $ and $X$ is any of the spaces $ l_{\infty},l_{1},c$ and $c_{0}.$
\par We use the following lemma to prove our results.
\begin{lem}\cite{A}
Let $ C=(c_{nk}) $ be defined via a sequence $ a=(a_{k}) \in \omega$ and the inverse matrix $ V=(v_{nk})$ of the triangle matrix $ U=(u_{nk})$ by 
 $$ c_{nk}
= \left\{
        \begin{array}{ll}
           \sum\limits_{j=k}^{n}a_{j}v_{jk} , 
           & 0 \leq k \leq n \\
            0  ,
           &  k > n
        \end{array}
  \right. $$
 for all $k,n \in \mathbb{N}.$ Then for any sequence space $\lambda,$\\
 $\lambda_{U}^{\gamma}=\left\lbrace a=(a_{k})\in \omega : C \in (\lambda, l_{\infty}) \right\rbrace $ and
 $\lambda_{U}^{\beta}=\left\lbrace a=(a_{k})\in \omega : C \in (\lambda, c) \right\rbrace .$
 \end{lem}
 \begin{thm}\label{thm42}
 Let $ \lambda = l_{p},$ $ 1\leq p \leq \infty$ and $\mu$ be an arbitrary subset of $\omega.$ Then  $A=(a_{nk}) \in (\lambda_{\hat{F}(r,s)}, \mu)$ if and only if \\
 \begin{equation}
 D^{(m)}=\left( d_{nk}^{(m)} \right)\in (\lambda,c) \,for\, all\, n \in \mathbb{N},  
 \label{41}
 \end{equation}
  \begin{equation}
  D=\left( d_{nk} \right)\in (\lambda,\mu), 
  \label{42}
  \end{equation}
  where 
  $$
  d_{nk}^{(m)}
  = \left\{
          \begin{array}{ll}
             \sum\limits_{j=k}^{m}\frac{1}{r}
             \left(-\frac{s}{r} \right)^{j-k} \frac{f_{j+1}^{2}}{f_{k}f_{k+1}}a_{nj} , 
             & 0\leq k \leq m \\
              0 ,
             & k>m
          \end{array}
      \right.
  $$
  and $d_{nk}=\sum\limits_{j=k}^{\infty}\frac{1}{r}
  \left(-\frac{s}{r}\right)^{j-k} \frac{f_{j+1}^{2}}{f_{k}f_{k+1}}a_{nj}$ for all $k,m,n \in \mathbb{N}.$
 \end{thm}
\begin{proof}
To prove this theorem, we follow the similar way due to Kiri\c{s}\c{c}i and Ba\c{s}ar \cite{Kiri}.
Let $A=(a_{nk}) \in (\lambda_{\hat{F}(r,s)}, \mu)$ and $ x=(x_{k})\in \lambda_{\hat{F}(r,s)}.$ We have from (\ref{24}),\\ $x_{k}=\sum\limits_{j=0}^{k}\frac{1}{r} \left(-\frac{s}{r}\right)^{k-j}\frac{f_{k+1}^{2}}{f_{j}f_{j+1}}y_{j}$ for all $k \in \mathbb{N}.$\\
 From (\ref{32}) we get
\begin{equation}
\sum\limits_{k=0}^{m} a_{nk}x_{k}\\
=\sum\limits_{k=0}^{m}\left( \sum\limits_{j=k}^{m}\frac{1}{r} \left(-\frac{s}{r}\right)^{j-k}\frac{f_{j+1}^{2}}{f_{k}f_{k+1}}a_{nj} \right)y_{k} \\ =\sum\limits_{k=0}^{m}d_{nk}^{(m)}y_{k}
=D_{n}^{(m)}(y),
\label{43}
\end{equation}
for all $m,n \in \mathbb{N}.$\\
 Since $Ax$ exists, $D^{(m)}\in (\lambda,c).$ As $m\rightarrow \infty$ in the equality (\ref{43}), we obtain  $Ax=Dy$ which implies $D \in (\lambda,\mu).$
\par Conversely, suppose (\ref{41}) and (\ref{42}) holds and take any $ x=(x_{k})\in \lambda_{\hat{F}(r,s)}.$ Then we have $ (d_{nk})\in \lambda^{\beta}$ which gives together with (\ref{41}) that $A_{n}=(a_{nk})_{k \in \mathbb{N}}\in  \lambda_{\hat{F}(r,s)}^{\beta}$ for all $n \in \mathbb{N}.$ Thus $Ax$ exists. Therefore we derive by equality (\ref{43}) as $m \rightarrow \infty$ that $Ax=Dy$ and this shows that
$A \in (\lambda_{\hat{F}(r,s)}, \mu).$
\end{proof} 
 Now we consider the following conditions  
 \begin{equation}
 \sup_{n}  \sum\limits_{k}\left| d_{nk}^{(m)}\right|^{q} < \infty , q=\frac{p}{p-1}
 \label{44}
 \end{equation} 
 \begin{equation}
 \lim_{n}d_{nk}^{(m)}  \mbox{~exists~} \forall \,k
 \label{45}
 \end{equation}
 \begin{equation}
 \lim_{n}\sum\limits_{k} \left|d_{nk}^{(m)}\right|=
 \sum\limits_{k} \left|\lim_{n} d_{nk}^{(m)}\right|
 \label{46}
 \end{equation}
  \begin{equation}
  \sup_{n}  \sum\limits_{k}\left| d_{nk}\right|^{q} < \infty , q=\frac{p}{p-1}
  \label{47}
  \end{equation}
 \begin{equation}
 \lim_{n}d_{nk} \mbox{~ exists~} \forall  \, k
 \label{48}
 \end{equation}
 \begin{equation}
 \lim_{n} \sum\limits_{k}\left|d_{nk}\right|
 = \sum\limits_{k}\left|\lim_{n}d_{nk}\right|
 \label{49}
 \end{equation}
 \begin{equation}
 \sup_{K \in \mathcal{F}} \sum\limits_{k}\left| \sum\limits_{n \in K} d_{nk}\right|^{q}  < \infty, q=\frac{p}{p-1}
 \label{410}
 \end{equation}
      
  Combining Theorems \ref{thm42} and Lemma \ref{lem11}, we derive the following results:\\
  \begin{cor}
  Let $A=(a_{nk})$ be an infinite matrix. Then the following statements hold:
  \begin{enumerate}
  \item[(a)] $A \in (l_{p}(\hat{F}(r,s)),c), 1<p<\infty$ if and only if (\ref{44}),(\ref{45}),(\ref{47}),(\ref{48}). 
  \item[(b)] $A \in (l_{p}(\hat{F}(r,s)),l_{1}),1<p<\infty$ if and only if (\ref{44}),(\ref{45}),(\ref{410}).
   \item[(c)] $A \in (l_{\infty}(\hat{F}(r,s)),c)$ if and only if    (\ref{45}),(\ref{46}),(\ref{48}),(\ref{49}).
  \item[(d)] $A \in (l_{p}(\hat{F}(r,s)),l_{\infty}),1<p<\infty$ if and only if
  (\ref{44}),(\ref{45}),(\ref{47}).
  \item[(e)] $A \in (l_{\infty}(\hat{F}(r,s)),l_{1})$ if and only if
  (\ref{45}),(\ref{46}),(\ref{410}).
  \item[(f)] $A \in (l_{\infty}(\hat{F}(r,s)),l_{\infty})$ if and only if    
  (\ref{45}),(\ref{46}),(\ref{47}).
  \end{enumerate}
  \end{cor}  
  \section{ Some geometric properties of the space $l_{p}(\hat{F}(r,s))\,(1<p<\infty)$}
  In this section, we study some geometric properties of the space $l_{p}(\hat{F}(r,s))\,(1<p<\infty).$
  \par
  For geometric properties we refer to \cite{Kara,G2,K1}.
  \par
  A Banach space $X$ is said to have the Banach-Saks property if every bounded sequence $(x_{n})$ in $X$ admits a subsequence $(z_{n})$ such that the sequence $\left\lbrace t_{k}(z) \right\rbrace $ is convergent in the norm in $X$ (see \cite{M1}), where
  \begin{equation}
  t_{k}(z)=\frac{1}{k+1}\left( z_{0}+z_{1}+...+z_{k}\right) 
  \: (k \in \mathbb{N})
  \label{51}
  \end{equation} 
  \par 
  A Banach space $X$ is said to have the weak Banach-Saks property whenever, given any weakly null sequence $(x_{n})\subset X $, there exists a subsequence $(z_{n})$ of $(x_{n})$ such that the sequence $\left\lbrace t_{k}(z) \right\rbrace $
  is strongly convergent to zero.
  \par
  In \cite{G2}, Garc\'{i}a-Falset introduced the following coefficient:
  \begin{equation}
  R(X)=\sup\left\lbrace 
    \liminf_{n \rightarrow \infty}  \parallel x_{n}-x\parallel : (x_{n})\subset B(X), 
    x_{n} \rightarrow 0(weakly), x \in  B(X) \right\rbrace    
    \label{52}
  \end{equation} 
  where $B(X)$ denotes the unit ball of $X.$
  \begin{rem} \cite{G2}\label{rem51}
  A Banach space $X$ with $R(X)<2$ has the weak fixed point property.
  \end{rem}
  
 Let $1<p<\infty.$ A Banach space is said to have the Banach-Saks type $p$ or the property $(BS)_{p}$ if every weakly null sequence $(x_{k})$ has a subsequence $\left( x_{k_{l}} \right)$ such that for some $C>0,$
 \begin{equation}
 \parallel \sum\limits_{l=0}^{n} x_{k_{l}} \parallel
 < C(n+1)^{1/p}
 \label{53}
 \end{equation}
 for all $n\in \mathbb{N}$ (see \cite{K1}).
 \par Now we are going to prove some geometric properties of the space 
 $l_{p}(\hat{F}(r,s))$ for $1<p<\infty.$
 \begin{thm}
 Let $1<p<\infty.$ Then the space $l_{p}(\hat{F}(r,s))$ has the Bnach-Saks type $p.$
 \end{thm}
 \begin{proof}
 Let $(\epsilon_{n})$ be a sequence of positive numbers for which $\sum \epsilon_{n} \leq 1/2,$ and also let $(x_{n})$ be a weakly null sequence in $B(l_{p}(\hat{F}(r,s))).$ Set $z_{0}=x_{0}=0$ and $z_{1}=x_{n_{1}}=x_{1}.$ Then there exists $m_{1} \in \mathbb{N}$ such that
 \begin{equation}
 \Arrowvert \sum\limits_{i=m_{1}+1}^{\infty} z_{1}(i)e^{(i)}\Arrowvert_{l_{p}(\hat{F}(r,s))}
 < \epsilon_{1}
 \label{54}
 \end{equation}
 Since $(x_{n})$ being a weakly null sequence implies $x_{n}\rightarrow 0$ coordinatewise, there is an $n_{2}\in \mathbb{N}$ such that
 \[\Arrowvert \sum\limits_{i=0}^{m_{1}} x_{n}(i)e^{(i)}\Arrowvert_{l_{p}(\hat{F}(r,s))}
  < \epsilon_{1},  \mbox{~when~} n \geq n_{2}.\]
   Set $z_{2}=x_{n_{2}}.$ Then there exists an $m_{2}>m_{1}$ such that
  \[ \Arrowvert \sum\limits_{i=m_{2}+1}^{\infty} z_{2}(i)e^{(i)}\Arrowvert_{l_{p}(\hat{F}(r,s))}
   < \epsilon_{2}.\]
   Again using the fact that $x_{n}\rightarrow 0$ coordinatewise, there exists an $n_{3} \geq n_{2}$ such that 
   \[\Arrowvert \sum\limits_{i=0}^{m_{2}} x_{n}(i)e^{(i)}\Arrowvert_{l_{p}(\hat{F}(r,s))}
     < \epsilon_{2}, \mbox{~when~} n\geq n_{3}.\]
 If we continue this process, we can find two increasing subsequences $(m_{i})$ and $(n_{i})$ such that 
   \[\Arrowvert \sum\limits_{i=0}^{m_{j}} x_{n}(i)e^{(i)}\Arrowvert_{l_{p}(\hat{F}(r,s))}
       < \epsilon_{j} \mbox{~for~ each~} n \geq n_{j+1}\] and
   \[\Arrowvert \sum\limits_{i=m_{j}+1}^{\infty} z_{j}(i)e^{(i)}\Arrowvert_{l_{p}(\hat{F}(r,s))}
        < \epsilon_{j}, \mbox{~where~} z_{j}=x_{n_{j}}.\]
   Hence \\
   $\Arrowvert \sum\limits_{j=0}^{n} z_{j}\Arrowvert_{l_{p}(\hat{F}(r,s))}\\
   =\Arrowvert \sum\limits_{j=0}^{n} \left( \sum\limits_{i=0}^{m_{j-1}}z_{j}(i)e^{(i)} +
   \sum\limits_{i=m_{j-1}+1}^{m_{j}}z_{j}(i)e^{(i)} + \sum\limits_{i=m_{j}+1}^{\infty}z_{j}(i)e^{(i)}
   \right)\Arrowvert_{l_{p}(\hat{F}(r,s)) }\\
   \leq \Arrowvert \sum\limits_{j=0}^{n} \left( \sum\limits_{i=m_{j-1}+1}^{m_{j}}z_{j}(i)e^{(i)} \right) \Arrowvert_{l_{p}(\hat{F}(r,s)) }
   +2\sum\limits_{j=0}^{n}\epsilon_{j}
   .$          
   \\ Since $z \in l_{p}(\hat{F}(r,s))$ therefore there exists $C>0$ such that 
   $\Arrowvert z\Arrowvert_{l_{p}(\hat{F}(r,s))} \leq C.$\\ Therefore we have that\\
   $\Arrowvert \sum\limits_{j=0}^{n} \left( \sum\limits_{i=m_{j-1}+1}^{m_{j}}z_{j}(i)e^{(i)} \right) \Arrowvert_{l_{p}(\hat{F}(r,s))}\\
   \leq \sum\limits_{j=0}^{n}\sum\limits_{i=m_{j-1}+1}^{m_{j}}\left| r \frac{f_{i}}{f_{i+1}}z_{j}(i)+
    s \frac{f_{i+1}}{f_{i}}z_{j}(i-1)\right|^{p}\\
    \leq \sum\limits_{j=0}^{n}\sum\limits_{i=0}^{\infty}\left| r \frac{f_{i}}{f_{i+1}}z_{j}(i)+
       s \frac{f_{i+1}}{f_{i}}z_{j}(i-1)\right|^{p}\\
    \leq  \sum\limits_{j=0}^{n} \Arrowvert z \Arrowvert_{l_{p}(\hat{F}(r,s))}
    \\ \leq C^{p}(n+1).$
    \\ Hence we obtain\\
    $\Arrowvert \sum\limits_{j=0}^{n} \left( \sum\limits_{i=m_{j-1}+1}^{m_{j}}z_{j}(i)e^{(i)} \right) \Arrowvert_{l_{p}(\hat{F}(r,s))}
    \leq C(n+1)^{p}.$\\
    By using the fact that $1 \leq (n+1)^{1/p}$ for all $n \in \mathbb{N}$ and $1<p<\infty,$ we have\\
    $\Arrowvert \sum\limits_{j=0}^{n} z_{j}\Arrowvert_{l_{p}(\hat{F}(r,s))}
    \leq C(n+1)^{p}+1 \leq (C+1)(n+1)^{p}.$\\
    Hence $l_{p}(\hat{F}(r,s))$ has the Banach-Saks type $p.$
    \begin{rem}\label{rem53}
    Note that $R\left(l_{p}(\hat{F}(r,s)) \right)=
    R\left(l_{p}\right)=2^{1/p} $ since $l_{p}(\hat{F}(r,s))$ is linearly isomorphic to $l_{p}.$
    \end{rem}
    By Remarks \ref{rem51} and \ref{rem53}, we have the following theorem.
    \begin{thm}
    The space $l_{p}(\hat{F}(r,s))$ has the weak fixed point property, where $1<p<\infty.$
    \end{thm}
 
 \end{proof}

\end{document}